\documentclass[12pt,a4paper]{article}

\usepackage{amsthm,amsmath,natbib,amssymb,amsfonts,bm}
\usepackage{graphicx}
\usepackage{placeins}

\def\hang{\hangindent\parindent}
\def\rf{\par\noindent\hang}

\newtheorem{theorem}{Theorem}
\newtheorem{lemma}{Lemma}

\newtheorem*{theorem*}{Theorem}

\theoremstyle{definition}

\theoremstyle{remark}

\DeclareMathAlphabet{\mathpzc}{OT1}{pzc}{m}{it}

\newcommand{\SSB}{\text{\small SSB}}
\newcommand{\SSW}{\text{\small SSW}}

\def\hang{\hangindent\parindent}
\def\rf{\par\noindent\hang}

\begin{document}

\baselineskip=20pt

\begin{center}
{\bf \Large Conditional assessment of the impact of a Hausman pretest on confidence intervals}
\end{center}



\bigskip

\begin{center}
{\bf \large Paul Kabaila$^*$, Rheanna Mainzer and Davide Farchione}
\end{center}

\medskip

\begin{center}
{\sl Department of Mathematics and Statistics, La Trobe University, Australia}
\end{center}

\vspace{1cm}

\noindent{\sl JEL classification:} \newline
\noindent C18 \newline
\noindent C23 \newline
\noindent C52

\vspace{1cm}

\noindent {\sl Keywords:} \newline
\noindent Conditional inference \newline
\noindent Coverage probability \newline
\noindent Fixed effects model \newline
\noindent Hausman specification test \newline
\noindent Panel data \newline
Random effects model

\vspace{7.5cm}

\noindent * Corresponding author. Department of Mathematics and Statistics,
La Trobe University, Victoria 3086, Australia. Tel.: +61 3 9479 2594; fax +61 3 9479 2466.
{\sl E-mail address:} P.Kabaila@latrobe.edu.au.

\newpage

\begin{center}
{\bf ABSTRACT}
\end{center}

We assess the impact of a Hausman pretest, applied to panel data, on a confidence interval for the slope, {\sl conditional} on the observed values of the time-varying covariate. This assessment has the advantages that it
(a) relates to the values of this covariate at hand, (b) is valid irrespective of how this covariate is generated,
(c) uses finite sample results and (d) results in an assessment that is determined by the values
of this covariate and only 2 unknown parameters. Our conditional analysis shows that the confidence interval
 constructed after a Hausman pretest should not be used.

\newpage

\section{Introduction}

For a linear regression model, where the explanatory variables are observed values of random variables, it has long been recognised that,
under commonly-occurring circumstances, statistical inference should be carried out conditional on these observed values. Aldrich (2005) describes the seminal contributions of R. A. Fisher (starting in 1922) and M. S. Bartlett to the recognition of this requirement.
Adoption of this requirement has the great advantage that the statistical
inference is valid irrespective of how the explanatory variables are generated.
In the econometric literature, an early recognition of this advantage is provided by Koopmans (1937, pp 29 and 30).
A modern description of the justification for this requirement is given in Example 4.3 of Cox (2006). The same justification also applies to linear regression models that include random effects. In particular, this justification applies to the model for longitudinal data that we consider. Statistical inference should be carried out conditional on the observed values of the time-varying covariate.
The statistical inference that we consider is a confidence interval for the slope parameter.
This confidence interval is assessed by its coverage probability and its scaled expected length, where the scaling is with respect to
the expected length of the standard confidence interval with the same minimum coverage probability.

Our aim is to analyze the effect of a Hausman pretest on the coverage probability and scaled expected length of a confidence interval for the slope, {\sl conditional} on the observed values of the time-varying covariate. The four main advantages of this analysis are the following.
Firstly, our analysis relates to the values of the time-varying covariate at hand and not to some other values that might have occurred, but are known to not have occurred. Secondly, our analysis has the great advantage that it applies irrespective of the Data Generating Process (DGP) for the time-varying covariate. We do not need to either restrict to some particular DGP, such as a first order autoregression, or concern ourselves with the possible values of the parameters that describe the chosen DGP. Thirdly, our analysis is a finite sample analysis, so that it does not rely on approximations based on large sample results, whose accuracy can be difficult to ascertain in the context of real life sample sizes.
Fourthly, as we show, the conditional coverage and scaled expected length of the confidence interval for the slope, constructed after a Hausman pretest, are determined by the time-varying covariate and only 2 unknown parameters: $\gamma$ which is a scaled version of a non-exogeneity parameter and $\nu$ which is the ratio (variance of random effect) / (variance of the random error), where $\gamma \in \mathbb{R}$ and
$\nu \in (0, \infty)$.

Previous analyses of the effect of a Hausman pretest on either a hypothesis test (Guggenberger, 2010)
or a confidence interval (Kabaila, Mainzer and Farchione, 2015) for the slope parameter, have been carried out {\sl unconditionally}. These analyses are restricted to particular DGP's for the time-varying covariate. For example, Kabaila, Mainzer and Farchione (2015) 
consider two models for the correlation matrix of the time-varying covariate: compound symmetry and first order autoregression.
It can be shown that the finite sample {\sl unconditional} coverage and scaled expected length of the confidence interval for the slope, constructed after a Hausman pretest, are determined by 4 known quantities, the unknown parameters $\gamma$ and $\nu$ and also the correlation structure of the time-varying covariate (Kabaila, Mainzer and Farchione, 2015). This additional dependence on this correlation structure is problematic: we are required to assign not only a model for this structure but also plausible ranges of the parameters that describe this model.

By contrast, by Theorems 1 and 4 of the present paper, the finite sample {\sl conditional} coverage and scaled expected length of the confidence interval for the slope, constructed after a Hausman pretest, are determined by the time-varying covariate, either 2 (coverage) or 3 (scaled expected length) known quantities and only two unknown parameters $\gamma$ and $\nu$. The correlation structure of the time-varying covariate is irrelevant.
Let $CP(\gamma, \nu)$ denote the conditional coverage probability of this confidence interval for the slope.
Wooldridge (2013) provides a balanced panel data set, \textsl{airfare}, used in exercise 14 of Chapter 14.  Suppose we are interested in the relationship between \textsl{concen} (a measure of market share) and \textsl{lfare} (log fare).  We use this data to demonstrate the effect of the Hausman pretest on the coverage probability and scaled expected length of a confidence interval for the slope parameter, conditional on the observed values of \textsl{concen}, the time-varying covariate.

To assess $CP(\gamma, \nu)$ we could use the \textsl{confidence coefficient}.
Throughout the paper, we use $g$ and $n$ to denote variables taking values in $\mathbb{R}$ and
$(0, \infty)$, respectively.
The confidence coefficient is the infimum over both $g$ and $n$ of $CP(g, n)$. Irrespective of the values of $\gamma$ and $\nu$,
$CP(\gamma, \nu)$ is bounded below by the \textsl{confidence coefficient}.
To find the \textsl{confidence coefficient} for the \textsl{airline} data we use Theorem \ref{thm_cov_even} of Section \ref{cp_pretest_interval}, which states that $CP(\gamma, \nu)$ is an even function of the unknown parameter $\gamma$. We find that the minimum over $g$ of $CP(g, n)$ is an increasing function of $n$. A graph of this function is shown in Figure 1. All computational results reported in this paper were found using programs
written in {\tt R}.
The \textsl{confidence coefficient} of the \textsl{airfare} data is approximately 0.19, which is
the limit as $n$ approaches 0 of  $\min_{g} CP(g, n)$.

However, this \textsl{confidence coefficient} does not utilize the information provided by the data about the unknown parameter $\nu$.
If the data strongly contradicts a value of $\nu$ near 0 then the \textsl{confidence coefficient} is an excessively conservative assessment of
$CP(\gamma, \nu)$.
An estimate of $\nu$ from the \textsl{airfare} data is $\widehat{\nu} = 12.78$.
In Section \ref{sec_conf_int_nu} we describe an equi-tailed confidence interval for $\nu$.
Theorem \ref{thm_pivot} gives a pivotal quantity which is key in the construction of this confidence interval.
The 98\% equi-tailed confidence interval for $\nu$ from the \textsl{airfare} data is $[11.3976, 14.3829]$.
This confidence interval strongly contradicts a value of $\nu$ near 0.

We therefore propose the following new assessment of
$CP(\gamma, \nu)$.
This new assessment is an equi-tailed confidence interval for the minimum over $g$ of $CP(g, \nu)$.
For the \textsl{airfare} data, the 98\% equi-tailed confidence interval for this minimum over $g$ is $[0.8889, 0.9026]$.


\FloatBarrier

\begin{figure}[h]
\centering
\includegraphics[scale = 0.8]{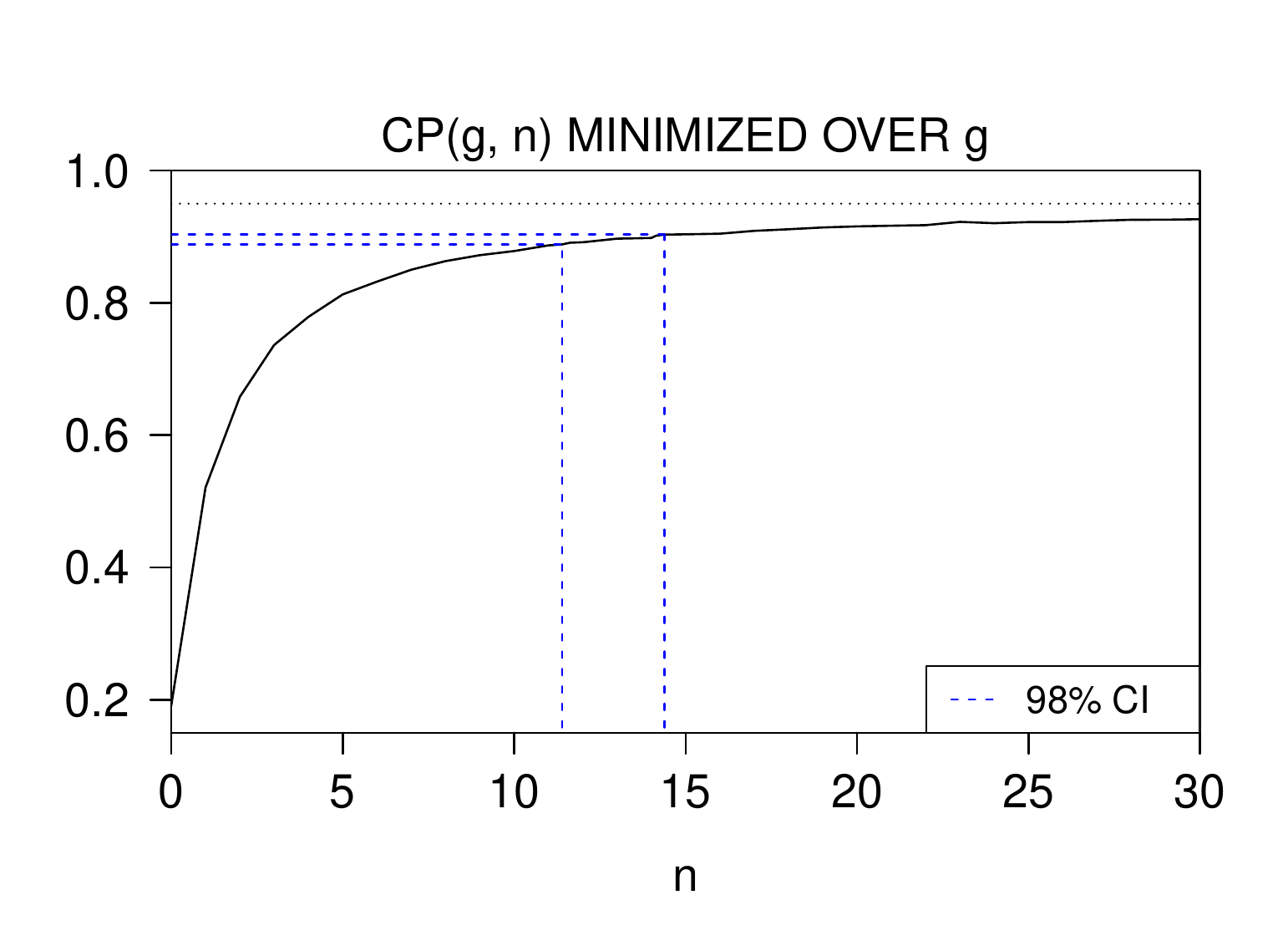}
\caption{For the \textsl{airfare} data, a
graph of $\min_{g} CP(g, n)$, as a function of $n$.
 The nominal significance level of the Hausman pretest is $\alpha_H = 0.05$ and the nominal conditional coverage probability of the confidence interval for the slope, constructed after the Hausman pretest, is $1 - \alpha = 0.95$.  A 98\% confidence interval for $\nu$ is given by the points where the dashed vertical lines intersect the horizontal axis and a 98\% confidence interval for $\min_{g} CP(g, \nu)$ is given by the points where the dashed horizontal lines intersect the vertical axis.}
\label{airfare_dat_plot}
\end{figure}

\FloatBarrier



\section{The model and the practical two-stage \newline procedure}

We consider a model for longitudinal data, for which $i$ denotes the individual ($i=1, \dots, N$) and $t$ denotes the time ($t=1, \dots, T$).  By interpreting $i$ as the cluster index and $t$ as the unit of analysis, our results also apply to the analysis of clustered data.  Let $y_{it}$ and $x_{it}$ denote the response variable and the time-varying covariate, respectively, for the $i$'th individual at time $t$.  Let $x = (x_{11}, \dots,  x_{1T}, \dots, x_{N1}, \dots, x_{NT})$.  Our statistical analysis is {\sl conditional} on the observed value of $x$, so that we treat $x$ as given.
Suppose that
\begin{equation}
\label{random_effects_model_eta}
y_{it} = a + b \, x_{it} + \xi \, \overline{x}_i + \eta_i + \varepsilon_{it}.
\end{equation}
for $i = 1, \dots, N$ and $t = 1, \dots, T$, where $\overline{x}_i = T^{-1} \sum_{t=1}^T x_{it}$.
Also suppose that the $\eta_i$'s and the $\varepsilon_{it}$'s are independent, with the $\eta_i$'s iid $N(0, \sigma_{\eta}^2)$ and the
$\varepsilon_{it}$'s iid $N(0, \sigma_{\varepsilon}^2)$. The $\varepsilon_{it}$'s and $\eta_i$'s are unobserved.  
This is the ``correlated random effects'' model described, for example, by Wooldridge (2013).
If $\xi = 0$ then the $x_{it}$'s are exogenous.
Thus we call $\xi$ a non-exogeneity parameter.

Suppose that the inference of interest is a confidence interval for the slope parameter $b$ with coverage probability $1-\alpha$,
conditional on $x$.
Assume, for the moment, that $\sigma_{\varepsilon}$ and $\sigma_{\eta}$ are known.
When $\xi = 0$ we can estimate $b$ efficiently from \eqref{random_effects_model_eta} using GLS.  Let $\widehat{b}$ denote this GLS estimator of $b$.
Also let $\nu = \sigma^2_{\eta} / \sigma^2_{\varepsilon}$ and $z_c = \Phi^{-1}(c)$, where $\Phi$ denotes the $N(0, 1)$ cdf.
A confidence interval for $b$ that has coverage probability $1-\alpha$ conditional on $x$, when $\xi = 0$, is
\begin{equation*}
I(\sigma_{\varepsilon}, \nu) = \left[ \widehat{b} - z_{1-\alpha/2} \, \left( \text{Var}(\widehat{b} \, | \, x ) \right)^{1/2} , \, \widehat{b} + z_{1 - \alpha/2} \, \left( \text{Var}(\widehat{b} \, | \, x) \right)^{1/2} \right],
\end{equation*}
where $\text{Var}(\widehat{b} \, | \, x)$ denotes the variance of $\widehat{b}$, conditional on $x$.
Adding and subtracting $b \, \overline{x}_i$ to \eqref{random_effects_model_eta} gives
\begin{align}
y_{it}
= a + b_W \, (x_{it} - \overline{x}_i) + b_B \, \overline{x}_i + \eta_i + \varepsilon_{it}, \label{CRE_model}
\end{align}
where $b_W = b$ and $b_B = b + \xi$.
Note that $b_W$ is the coefficient of a ``within'' effect and $b_B$ is the coefficient of a ``between'' effect. 
Let $\widetilde{b}_W$ and $\widetilde{b}_B$ denote the GLS estimators of $b_W$ and $b_B$, respectively, from model \eqref{CRE_model}.
A confidence interval for $b$ that has coverage probability $1 - \alpha$ conditional on $x$, irrespective of the value of $\xi$, is
\begin{equation*}
J(\sigma_{\varepsilon}) = \left[ \widetilde{b}_W - z_{1-\alpha/2} \, \left( \text{Var}(\widetilde{b}_W \, | \, x ) \right)^{1/2} , \, \widetilde{b}_W + z_{1-\alpha/2} \, \left( \text{Var}(\widetilde{b}_W \, | \, x ) \right)^{1/2} \right],
\end{equation*}
where $\text{Var}(\widetilde{b}_W \, | \, x)$ denotes the variance of $\widetilde{b}_W$, conditional on $x$.

As proved in the appendix, $\widetilde{b}_W$ and $\widetilde{b}_B$ can also be obtained as follows.
Averaging \eqref{CRE_model} over $t=1, \dots, T$ for each $i=1, \dots, N$, we obtain the model
\begin{equation}
\label{BE_model}
\overline{y}_i = a + b_B \, \overline{x}_i + \eta_i + \overline{\varepsilon}_i.
\end{equation}
where $\overline{y}_i = T^{-1} \sum_{t=1}^T y_{it}$ and $\overline{\varepsilon}_i = T^{-1} \sum_{t=1}^T \varepsilon_{it}$.
The OLS estimator of $b_B$ based on this model is equal to $\widetilde{b}_B$.
Subtracting \eqref{BE_model} from \eqref{CRE_model} we obtain the ``fixed effects'' model
\begin{equation}
\label{FE_model}
y_{it} - \overline{y}_i = b_W \, (x_{it} - \overline{x}_i) + (\varepsilon_{it} - \overline{\varepsilon}_i).
\end{equation}
The OLS estimator of $b_W$ based on this model is equal to $\widetilde{b}_W$.

In practice we do not know whether $\xi = 0$ or not.  The usual procedure is to use a Hausman (1978) pretest to test the null hypothesis that $\xi = 0$ ($b_W = b_B$) against the alternative hypothesis that $\xi \not= 0$ ($b_W \not = b_B$).
 We consider this pretest, based on the test statistic
\begin{equation*}
H(\sigma_{\varepsilon}, \nu) = \frac{ (\widetilde{b}_W - \widetilde{b}_B)^2}{\text{Var}(\widetilde{b}_W \, | \, x) + \text{Var}(\widetilde{b}_B \, | \, x)},
\end{equation*}
where $\text{Var}(\widetilde{b}_B \, | \, x)$ denotes the variance of $\widetilde{b}_B$, conditional on $x$.  This test statistic has a $\chi^2_1$ distribution under the null hypothesis.  Suppose that we accept the null hypothesis that $\xi = 0$ if $H(\sigma_{\varepsilon}, \nu) \leq z^2_{1-\alpha_H/2}$; otherwise we reject the null hypothesis.  The level of significance of this test is $\alpha_H$.
We consider the following two-stage procedure.  If the null hypothesis is accepted then use the confidence interval $I(\sigma_{\varepsilon}, \nu)$; otherwise use the confidence interval $J(\sigma_{\varepsilon})$.  Let $K(\sigma_{\varepsilon}, \nu)$ denote the confidence interval, with nominal coverage $1 - \alpha$, that results from this two-stage procedure.

Of course, in practice, $\sigma_{\varepsilon}$ and $\nu$ need to be estimated from the data.  Let $\widehat{\sigma}_{\varepsilon}$ and $\widehat{\nu}$ denote estimators of $\sigma_{\varepsilon}$ and $\nu$, respectively, described in Section \ref{EstimationVariances}.  Let $H(\widehat{\sigma}_{\varepsilon}, \widehat{\nu})$, $I(\widehat{\sigma}_{\varepsilon}, \widehat{\nu})$, $J(\widehat{\sigma}_{\varepsilon})$ and $K(\widehat{\sigma}_{\varepsilon}, \widehat{\nu})$ denote the Hausman test statistic, the confidence interval based on $\widehat{b}$, the confidence interval based on $\widetilde{b}_W$ and the confidence interval that results from the two-stage procedure, respectively, when $\sigma_{\varepsilon}$ and $\nu$ are replaced by their estimators. Our aim is to assess the coverage probability and expected length properties of $K(\widehat{\sigma}_{\varepsilon}, \widehat{\nu})$, {\sl conditional} on $x$.

\subsection{Estimation of $\boldsymbol{\sigma_{\varepsilon}}$ and $\boldsymbol{\nu}$}
\label{EstimationVariances}

We use the models \eqref{BE_model} and \eqref{FE_model}
to motivate the estimators of $\sigma^2_{\varepsilon}$ and $\nu$ that we use.
Let $r_{it}$ denote the residual for the $i$'th individual at the $t$'th time when \eqref{FE_model} is estimated by OLS, i.e. let $r_{it} = (y_{it} - \overline{y}_i) - \widetilde{b}_W \, (x_{it} - \overline{x}_i)$ for $i=1, \dots, N$ and $t=1, \dots, T$.  Then $r_{it} =  (b_W - \widetilde{b}_W) \, (x_{it} - \overline{x}_i) + (\varepsilon_{it} - \overline{\varepsilon}_i)$. Our estimator
\begin{equation}
\label{est_var_varepsilon}
\widehat{\sigma}^2_{\varepsilon} = \frac{1}{N(T-1)} \sum_{i=1}^N \sum_{t=1}^{T} r_{it}^2.
\end{equation}
is motivated by the approximation $r_{it} \approx \varepsilon_{it} - \overline{\varepsilon}_i$ and the fact that
$E(\varepsilon_{it} - \overline{\varepsilon}_i) = 0$ and $\text{Var}(\varepsilon_{it} - \overline{\varepsilon}_i) = N(T-1) \sigma^2_{\varepsilon}$.

Now we turn to the estimation of $\sigma^2_{\eta}$.
The most efficient estimator of $(a, b_B)$, based on the model \eqref{BE_model},
is the OLS estimator $(\widetilde{a}, \widetilde{b}_B)$.  Let $\widetilde{r}_i$ denote the $i$'th residual corresponding to this estimator.
In other words, $\widetilde{r}_i = \overline{y}_i - (\widetilde{a} + \widetilde{b}_B \, \overline{x}_i) = (a - \widetilde{a}) + (b_B - \widetilde{b}_B) \, \overline{x}_i + \eta_i + \overline{\varepsilon}_i$.
Our estimator
\begin{equation}
\label{est_var_eta}
\widehat{\sigma}^2_{\eta} = N^{-1} \sum_{i=1}^N \widetilde{r}_i^2 - T^{-1} \, \widehat{\sigma}^2_{\varepsilon}.
\end{equation}
is motivated by the approximation $\widetilde{r}_i \approx \eta_i + \overline{\varepsilon}_i$
and the fact that
$\left( \eta_i+ \overline{\varepsilon}_i \right)$'s are independent, with
$E\left( \eta_i+ \overline{\varepsilon}_i \right) = 0$ and $\text{Var}\left(\eta_i+ \overline{\varepsilon}_i \right) = \sigma^2_{\eta} + T^{-1} \, \sigma^2_{\varepsilon}$. Our estimator of $\nu$ is $\widehat{\nu} = \widehat{\sigma}_{\eta}^2 / \widehat{\sigma}_{\varepsilon}^2$.

\section{The coverage probability of the confidence interval resulting from the two-stage procedure}
\label{cp_pretest_interval}

The coverage probability of the confidence interval $K(\widehat{\sigma}_{\varepsilon}, \widehat{\nu})$, conditional on $x$, is $P(b \in K(\widehat{\sigma}_{\varepsilon}, \widehat{\nu}) \, | \, x)$.  By the law of total probability, $P(b \in K(\widehat{\sigma}_{\varepsilon}, \widehat{\nu}) \, | \, x )$ is equal to
\begin{equation}
\label{cov_prob_unknown}
 P\left( b \in I(\widehat{\sigma}_{\varepsilon}, \widehat{\nu}), \, H(\widehat{\sigma}_{\varepsilon}, \widehat{\nu}) \leq z^2_{1-\alpha_H/2} \, | \, x \right) + P \left( b \in J(\widehat{\sigma}_{\varepsilon}), \, H(\widehat{\sigma}_{\varepsilon}, \widehat{\nu}) > z^2_{1-\alpha_H/2} \, | \, x \right).
\end{equation}
Let $\gamma = \xi \, N^{1/2} / \sigma_{\varepsilon}$, which is a
scaled version of the non-exogeneity parameter $\xi$.
The following two theorems give important properties of this coverage probability.  The proofs of these theorems are in the appendix.
\begin{theorem}
\label{thm_cov_depends}
For the estimators considered in Section \ref{EstimationVariances}, $P(b \in K(\widehat{\sigma}_{\varepsilon}, \widehat{\nu}) \, | \, x)$ is determined by $x$ (the time-varying covariate), $\alpha_H$ (the nominal significance level of the Hausman pretest), $1 - \alpha$ (the nominal coverage probability of $K(\widehat{\sigma}_{\varepsilon}, \widehat{\nu})$), $\nu$ (the ratio $\sigma^2_{\eta}/\sigma^2_{\varepsilon}$) and $\gamma$ (the scaled non-exogeneity parameter).  Given these quantities, the conditional coverage probability does not depend on $\sigma^2_{\varepsilon}$ (the variance of the random error) or $\sigma^2_{\eta}$ (the variance of the random effect).
\end{theorem}
\begin{theorem}
\label{thm_cov_even}
Suppose that $x$, $\alpha_H$, $1-\alpha$ and $\nu$ are fixed.  For the estimators considered in Section \ref{EstimationVariances}, the conditional coverage probability is an even function of $\gamma$.
\end{theorem}

\subsection{Assessment of $\boldsymbol{CP(\gamma, \nu)}$}

Suppose that $\alpha_H$ and $1-\alpha$ are given.
Let $CP(\gamma, \nu) = P(b \in K(\widehat{\sigma}_{\varepsilon}, \widehat{\nu}) \, | \, x)$,
the coverage probability of $K(\widehat{\sigma}_{\varepsilon}, \widehat{\nu})$, conditional on $x$.
In this section, we ask the question: How do we assess
$CP(\gamma, \nu)$?
A commonly used assessment of this coverage probability is the {\sl confidence coefficient}, which
is the infimum over $g$ and $n$ of $CP(g, n)$.
Irrespective of the values of $\gamma$ and $\nu$,
$CP(\gamma, \nu)$ is bounded below by the \textsl{confidence coefficient}.
We illustrate this assessment using the {\sl airfare} data.
For the \textsl{airfare} data, the minimum over $g$ of $CP(g, n)$ is an increasing function of $n$.
For this data, the \textsl{confidence coefficient} is approximately 0.19.

However, this assessment is excessively conservative, since the data provides information about
$\nu$ through the estimator $\widehat{\nu}$. For any given $x$, we can assess how much $CP(\gamma, \widehat{\nu})$
differs from $CP(\gamma, \nu)$ for a range of values of
$\gamma$ and $\nu$. This has been done for the {\sl airfare} data in Figure \ref{CP_pdf_gamtrue_nuhat}
by plotting the density function of
$CP(\gamma, \widehat{\nu})$, estimated by the simulation method described later in Section \ref{sec_est_pdfs_CP}.
What this figure shows us is that $CP(\gamma, \widehat{\nu})$ is unlikely to differ greatly from $CP(\gamma, \nu)$.
This suggests that we can provide a useful assessment of $CP(\gamma, \nu)$
by finding an equi-tailed confidence interval for the minimum over $g$ of $CP(g, \nu)$.
This confidence interval is constructed from the equi-tailed confidence interval for $\nu$
described in the next section. An attractive feature of this assessment is if (hypothetically)
the confidence interval for $\nu$ is $(0, \infty)$ then this assessment reduces to the confidence coefficient.

\begin{figure}
\centering
\includegraphics[scale = 0.9]{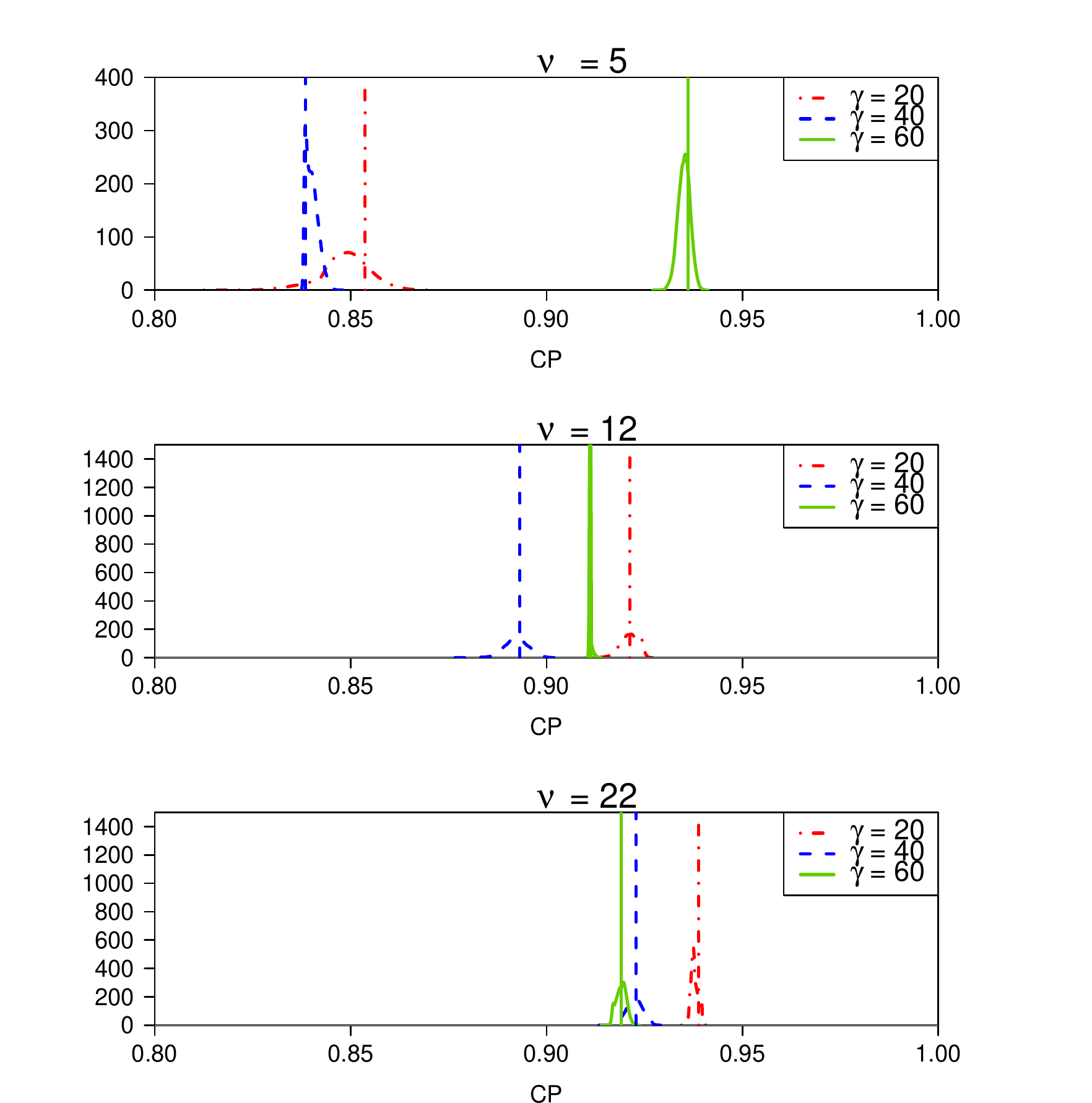}
\caption{For the {\sl airfare} data, plots of the density function, conditional on $x$, of
$CP(\gamma, \widehat{\nu})$, estimated by simulation, for $\nu \in \{ 5, 12, 22 \}$
and $\gamma \in \{20, 40, 60 \}$. The vertical lines have horizontal axis intercepts at $CP(\gamma, \nu)$.}
\label{CP_pdf_gamtrue_nuhat}
\end{figure}

\subsection{An equi-tailed confidence interval for $\boldsymbol{\nu}$}
\label{sec_conf_int_nu}

  Using \eqref{est_var_varepsilon} and \eqref{est_var_eta}, it can be shown that
\begin{align}
\label{nu_hat_over_nu}
\frac{\widehat{\nu} + T^{-1}}{\nu + T^{-1}} = \frac{\dfrac{ N^{-1} \sum_{i=1}^N \widetilde{r}_i^2}{\sigma^2_{\varepsilon}(\nu + T^{-1})}}{ \dfrac{\left(N(T-1)\right)^{-1} \sum_{i=1}^N \sum_{t=1}^T r_{it}^2}{\sigma^2_{\varepsilon}}}.
\end{align}
The following theorem allows us to easily compute quantiles of the distribution of $(\widehat{\nu} + T^{-1}) / (\nu + T^{-1})$ by simulation.  This theorem is proved in the appendix.
\begin{theorem}
\label{thm_pivot}
Conditional on $x$,
the distribution of $\left(\widehat{\nu} + T^{-1}\right) / \left( \nu + T^{-1} \right)$ does not depend on any unknown parameters, i.e. $\left(\widehat{\nu} + T^{-1}\right) / \left( \nu + T^{-1} \right)$ is a pivotal quantity.
\end{theorem}

Define $F_{\overline{\alpha}/2}$ and $F_{1 - \overline{\alpha}/2}$ to be the $\overline{\alpha}/2$ and $1 - \overline{\alpha}/2$ quantiles, respectively, of the distribution of the pivotal quantity $(\widehat{\nu} + T^{-1}) / (\nu + T^{-1})$.  Then an equi-tailed confidence interval for $\nu$ with coverage probability $1 - \overline{\alpha}$ is
\begin{equation}
\label{conf_int_nu}
 \left[  \frac{ \widehat{\nu} + T^{-1}}{F_{1-\overline{\alpha}/2}} - T^{-1} , \, \frac{\widehat{\nu} + T^{-1}}{F_{\overline{\alpha}/2}} - T^{-1} \right].
\end{equation}
We find $F_{\overline{\alpha}/2}$ and $F_{1 - \overline{\alpha}/2}$  by the simulation method described in detail in Section \ref{sec_est_ci_nu}.  For the \textsl{airfare} data described in the introduction, this 98\% equi-tailed confidence interval for $\nu$ is $[11.3976, 14.3829]$.

\bigskip

\subsection{Why don't we use information about $\boldsymbol{\gamma}$ provided by $\boldsymbol{\widehat{\gamma}}$?}

An estimator of $\gamma$ is
\begin{equation}
\label{eqn_gam_hat}
\widehat{\gamma} =  \frac{(\widetilde{b}_B - \widetilde{b}_W) \, N^{1/2}}{\widehat{\sigma}_{\varepsilon}}.
\end{equation}
The data provides information about $\gamma$ through the estimator $\widehat{\gamma}$.
For any given $x$, we can assess how much $CP(\widehat{\gamma}, \nu)$
differs from $CP(\gamma, \nu)$ for a range of values of
$\gamma$ and $\nu$. This has been done for the {\sl airfare} data in Figure \ref{CP_pdf_gamhat_nutrue}
by plotting the probability density function of
$CP(\widehat{\gamma}, \nu)$, estimated by the simulation method described later in Section \ref{sec_est_pdfs_CP}.
What this figure shows us is that $CP(\widehat{\gamma}, \nu)$ differs greatly from $CP(\gamma, \nu)$.
Therefore, we treat the parameter $\gamma$ differently from the parameter $\nu$.

\begin{figure}
\centering
\includegraphics[scale = 0.9]{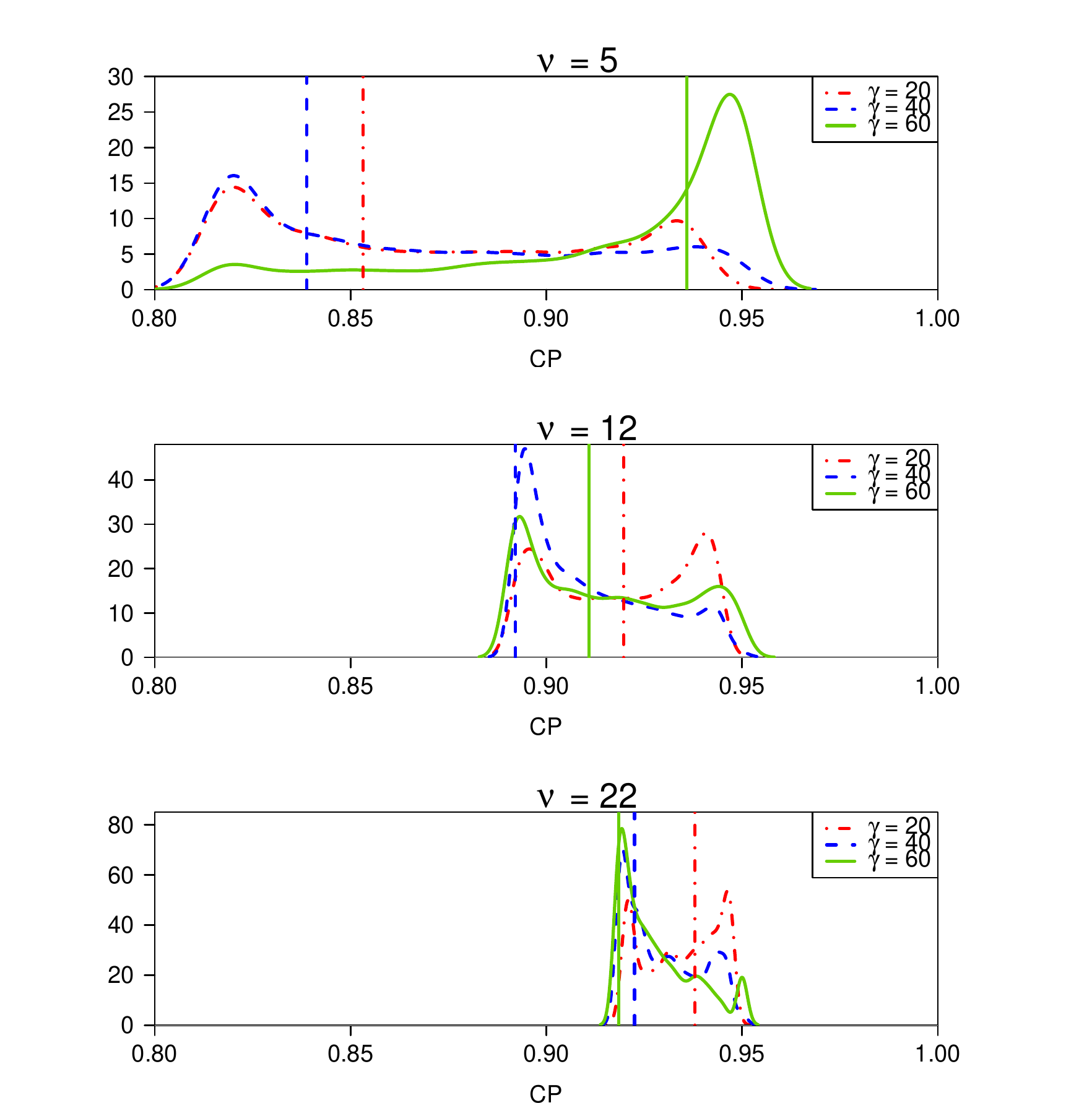}
\caption{For the {\sl airfare} data, plots of the density function, conditional on $x$, of
$CP(\widehat{\gamma}, \nu)$, estimated by simulation, for $\nu \in \{ 5, 12, 22 \}$
and $\gamma \in \{20, 40, 60 \}$.
The vertical lines have horizontal axis intercepts at $CP(\gamma, \nu)$.
}
\label{CP_pdf_gamhat_nutrue}
\end{figure}

\section{Definition of the conditional scaled expected length}
\label{sel_pretest_interval}

For any $c \in [1/2, 1)$, let
\begin{equation*}
J_c(\sigma_{\varepsilon}) = \left[ \widetilde{b}_W - \Phi^{-1}((c+1)/2)\left( \text{Var}(\widetilde{b}_W \, | \, x) \right)^{1/2},  \, \widetilde{b}_W + \Phi^{-1}((c+1)/2)\left( \text{Var}(\widetilde{b}_W \, | \, x) \right)^{1/2} \right].
\end{equation*}
Note that, for any given $c$, the confidence interval $J_c(\widehat{\sigma}_{\varepsilon})$ has coverage probability that does not depend on any unknown parameters.
This is the standard confidence interval against which we compare $K(\widehat{\sigma}_{\varepsilon}, \widehat{\nu})$, in terms of
expected length, conditional on $x$.
As observed in Section 2, this interval may be constructed using the fixed effects model.

The usual definition of conditional scaled expected length of $K(\widehat{\sigma}_{\varepsilon}, \widehat{\nu})$ is as follows.
Define $c_{\text{min}}$ to be the value of $c$ such that
$P(b \in J_c(\widehat{\sigma}_{\varepsilon})) = \inf_n \min_{g} CP(g, n)$. Then define this scaled expected length
to be the expected length of $K(\widehat{\sigma}_{\varepsilon}, \widehat{\nu})$ divided by the expected length of $J_{c_{\text{min}}}(\widehat{\sigma}_{\varepsilon})$, conditional on $x$. In other words, we compare the expected length of $K(\widehat{\sigma}_{\varepsilon}, \widehat{\nu})$ with the expected length of the standard confidence interval with the same {\sl confidence coefficient}, conditional on $x$.

However, as noted in the introduction, the \textsl{confidence coefficient} for the \textsl{airline} data is an excessively conservative
assessment of $CP(\gamma, \nu)$. We therefore
introduce the following alternative definition of conditional scaled expected length.
Let $[\nu_{\ell}, \nu_{u}]$ be the equi-tailed confidence interval \eqref{conf_int_nu} for $\nu$ with coverage probability $1 - \overline{\alpha}$.
Define $c^*$ to be the value of $c$ such that
$P(b \in J_c(\widehat{\sigma}_{\varepsilon})) = \min_{n \in [\nu_{\ell}, \nu_{u}]} \min_{g} CP(g, n)$.
The conditional scaled expected length of $K(\widehat{\sigma}_{\varepsilon}, \widehat{\nu})$ is defined to be the expected length of $K(\widehat{\sigma}_{\varepsilon}, \widehat{\nu})$ divided by the expected length of $J_{c^*}(\widehat{\sigma}_{\varepsilon})$,
conditional on $x$. An attractive feature of this definition is that if (hypothetically)
the equi-tailed confidence interval for $\nu$ with coverage probability $1 - \overline{\alpha}$ is $(0, \infty)$ then $c^*$ is equal to $c_{\text{min}}$, and this definition of the scaled expected length reduces to the usual definition.
Let $\overline{x} = N^{-1} \sum_{i=1}^N \overline{x}_i$, $\SSB = \sum_{i=1}^N (\overline{x}_i - \overline{x})^2$, $\SSW = \sum_{i=1}^N \sum_{t=1}^T (x_{it} - \overline{x}_i)^2$, $r(x) = \SSB / \SSW$, $q(\widehat{\nu}, T) = \widehat{\nu} + T^{-1}$, $\widehat{w} = q(\widehat{\nu}, T) / \left( q(\widehat{\nu}, T) + r(x) \right)$ and define the event
${\cal H} = \big\{ H(\widehat{\sigma}_{\varepsilon}, \widehat{\nu}) \leq z^2_{1 - \alpha_H/2} \big \}$.
For any statement $\mathcal{A}$, we use the notation
\begin{equation*}
\mathcal{I}(\mathcal{A}) = \begin{cases} 1 \text{ if $\mathcal{A}$ is true} \\ 0 \text{ if $\mathcal{A}$ is false} \end{cases}.
\end{equation*}
The following theorem gives a convenient expression for this alternative definition of the conditional scaled expected length of
$K(\widehat{\sigma}_{\varepsilon}, \widehat{\nu})$.  This theorem is proved in the appendix.

\begin{theorem}
\label{thm_sel}
For $\gamma \in \mathbb{R}$ and $\nu \in [\nu_l, \nu_u]$, the scaled expected length, conditional on $x$, is equal to
\begin{equation}
\label{scaled_expected_length}
\frac{ z_{1-\alpha/2}}{\Phi^{-1}\left( (c^* + 1) / 2 \right)} \frac{ E \left( \left( \widehat{\sigma}_{\varepsilon} / \sigma_{\varepsilon} \right) \left( \widehat{w}^{1/2} \, {\cal I}( {\cal H}) + {\cal I}( {\cal H}^c) \right)\, | \, x \right)}{E\left( \widehat{\sigma}_{\varepsilon} / \sigma_{\varepsilon} \right)}.
\end{equation}
\end{theorem}

The following two theorems give important properties of this scaled expected length.  These theorems are proved in the appendix.

\begin{theorem}
\label{thm_sel_depends}
For the estimators described in Section \ref{EstimationVariances} the scaled expected length, conditional on $x$, is determined by $x$ (the time-varying covariate), $\alpha_H$ (the nominal significance level of the Hausman pretest), $1 - \alpha$ (the nominal coverage probability of $K(\widehat{\sigma}_{\varepsilon}, \widehat{\nu})$), $1 - \overline{\alpha}$ (the coverage probability of the confidence interval for $\nu$), $\nu$ (the ratio $\sigma^2_{\eta} / \sigma^2_{\varepsilon}$) and $\gamma$ (the scaled non-exogeneity parameter).  Given these quantities, the conditional scaled expected length does not depend on $\sigma^2_{\varepsilon}$ (the variance of the random error) or $\sigma^2_{\eta}$ (the variance of the random effect).
\end{theorem}

\begin{theorem}
\label{thm_sel_even}
Suppose that $x$, $\alpha_H$, $1 - \alpha$, $1 - \overline{\alpha}$ and $\nu$ are fixed.  When $\sigma_{\varepsilon}$ and $\sigma_{\mu}$ are replaced by the estimators described in Section \ref{EstimationVariances}, the conditional scaled expected length is an even function of $\gamma$.
\end{theorem}

\subsection{Numerical illustration for the \textsl{airfare} data}

Let $SEL(\gamma, \nu)$ denote the conditional scaled expected length of
$K(\widehat{\sigma}_{\varepsilon}, \widehat{\nu})$.
For given $x$, $\alpha_H$, $1 - \alpha$ and $1 - \overline{\alpha}$, $SEL(g, \nu)$, minimized over $g$, is a function of $\nu$.
In this section we illustrate the application of $SEL(\gamma, \nu)$ using the \textsl{airfare} data.
We estimate the scaled expected length using the simulation method described in Section \ref{sec_est_sel} for $\alpha_H = 0.05$, $1 - \alpha = 0.95$, $1 - \overline{\alpha} = 0.98$ and $M = 50000$.

As stated in the introduction,
the 98\% equi-tailed confidence interval for $\nu$ is $[11.3976, 14.3829]$.
Using Theorem \ref{thm_sel}, this leads to the 98\% equi-tailed confidence interval for
$SEL(g, \nu)$, minimized over $g$, $[1.1012, 1.1244]$.
Of course, this confidence interval utilizes the information provided by the data about the unknown parameter $\nu$.
Using the usual definition of the scaled expected length that does not utilize this information, we find that $\inf_{n} \min_{g} SEL(g, n) = 2.5208$.
For both definitions of the scaled expected length, we have nothing to gain by using the confidence interval $K(\widehat{\sigma}_{\varepsilon}, \widehat{\nu})$.

%
%

\section{Simulation methods}

In Section \ref{sec_cp_est's_known} we give a new theorem that allows us to find a control variate for the estimation by simulation of
$CP(\gamma, \nu)$ and $SEL(\gamma, \nu)$
for any given $\gamma$ and $\nu$.  In Sections \ref{sec_est_cp} and \ref{sec_est_sel} we describe how to estimate $CP(\gamma, \nu)$ and $SEL(\gamma, \nu)$, respectively, by simulation.  We also describe how to make use of a control variate for variance reduction.
Section \ref{sec_est_pdfs_CP} describes the estimation of the density functions of $CP(\gamma, \widehat{\nu})$ and $CP(\widehat{\gamma}, \nu)$ shown in Figures \ref{CP_pdf_gamtrue_nuhat} and \ref{CP_pdf_gamhat_nutrue}, respectively, and Section \ref{sec_est_ci_nu} describes
the simulation method used to find the quantiles $F_{\overline{\alpha}/2}$ and $F_{1 - \overline{\alpha}/2}$ needed for the
 construction of the equi-tailed confidence interval for $\nu$ described in Section \ref{sec_conf_int_nu}.

\subsection{The two-stage procedure when $\boldsymbol{\sigma_{\varepsilon}}$ and $\boldsymbol{\nu}$ are known}
\label{sec_cp_est's_known}

In this section we give a new theorem that is used to find a control variate for the estimation by simulation of the coverage probability and scaled expected length of $K(\widehat{\sigma}_{\varepsilon}, \widehat{\nu})$.
For the rest of this section suppose that $\sigma_{\varepsilon}$ and $\nu$ are known.
Define the random variables
\begin{align*}
g_I = \frac{\widehat{b} - b}{\left(\text{Var}(\widehat{b} \, | \, x )\right)^{1/2}}, \, \, g_J = \frac{\widetilde{b}_W - b}{\left( \text{Var}(\widetilde{b}_W \, | \, x) \right)^{1/2}}, \, \,  h = \frac{ \widetilde{b}_W - \widetilde{b}_B}{\left( \text{Var}(\widetilde{b}_W \, | \, x) + \text{Var}(\widetilde{b}_B \, | \, x) \right)^{1/2}}.
\end{align*}
By the law of total probability $P\left(b \in K(\sigma_{\varepsilon}, \nu) \, | \, x \right)$, is equal to
\begin{align}
 & P\left( b \in I(\sigma_{\varepsilon}, \nu), \, H(\sigma_{\varepsilon}, \nu) \leq z^2_{1-\alpha_H/2} \, | \, x \right) + P \left( b \in J(\sigma_{\varepsilon}), \, H(\sigma_{\varepsilon}, \nu) > z^2_{1-\alpha_H/2} \, | \, x \right) \nonumber \\[1.2ex]
&= (1 - \alpha) + P(|g_I| \leq z_{1-\alpha/2}, \, |h| \leq z_{1-\alpha_H/2} \, | \, x ) - P(|g_J| \leq z_{1-\alpha/2}, \, |h| \leq z_{1-\alpha_H/2} \, | \, x). \label{cov_prob_known}
\end{align}
The coverage probability of $K(\sigma_{\varepsilon}, \nu)$ is determined by the distributions of the random vectors $(g_I, h)$ and $(g_J, h)$.
Define $q(\nu, T) = \nu + T^{-1}$ and $w = q(\nu, T) / (q(\nu, T) + r(x))$.
Theorem \ref{thm: dist_gih_gjh} gives the distributions of the random vectors $(g_I, h)$ and $(g_J, h)$.  This theorem is proved in the appendix.
\begin{theorem}
\label{thm: dist_gih_gjh}
Conditional on $x$, $(g_I, h)$ and $(g_J, h)$ have bivariate normal distributions where
\begin{align*}
&E(g_J \, | \, x) = 0, \, \, \, \, \, \, \, \text{Var}(g_J \, | \, x) = 1, \, \, \, \, \, \, \,
E(g_I \, | \, x) =\gamma \, \left( \frac{\SSB}{N} \right)^{1/2} \left( \frac{r(x)}{q(\nu, T) \left( q(\nu, T) + r(x) \right)}\right)^{1/2} , \\[1.2ex]
&\text{Var}(g_I \, | \, x) = 1, \, \, \, \, \, \, \, E(h \, | \, x) = -\gamma \, \left(\frac{\SSB}{N}\right)^{1/2} \left(\frac{ 1 }{ r(x) + q(\nu, T) } \right)^{1/2} , \, \, \, \, \, \, \, \text{Var}(h \, | \, x) = 1, \\[1.2ex]
&\text{Cov}(g_J, h \, | \, x ) = \left( \frac{r(x)}{ r(x) + q(\nu, T)} \right)^{1/2}, \, \, \, \, \, \, \, \text{Cov}(g_I, h \, | \, x) = 0.
\end{align*}
\end{theorem}

\subsection{Estimation of $\boldsymbol{CP(\gamma, \nu)}$ for given $\boldsymbol{\gamma}$ and $\boldsymbol{\nu}$ by simulation}
\label{sec_est_cp}

Consider a grid of $\gamma$ values.  Note that, by the proof of Theorem \ref{thm_cov_depends}, $\widehat{g}_I$, $\widehat{g}_J$ and $\widehat{h}$ can be expressed in terms of the $\varepsilon_{it}^{\dag}$'s and the $\eta_i^{\dag}$'s, where the  $\varepsilon_{it}^{\dag}$'s and $\eta_i^{\dag}$'s are iid $N(0, 1)$.  The simulation method consists of $M$ independent simulation runs.  On the $k$'th simulation run ($k=1, \dots, M$) we generate observations of the $\varepsilon_{it}^{\dag}$'s and the $\eta_i^{\dag}$'s and compute $\widehat{g}_I$, $\widehat{g}_J$ and $\widehat{h}$ for each $\gamma$ in the grid of values of $\gamma$.  Let the values of $\widehat{g}_I$, $\widehat{g}_J$ and $\widehat{h}$ on the $k$'th simulation run be denoted by $\widehat{g}_{I, k}$, $\widehat{g}_{J, k}$ and $\widehat{h}_k$, respectively.  Also let $\textsc{CP} = P(b \in K(\widehat{\sigma}_{\varepsilon}, \widehat{\nu}) | x)$, the coverage probability of $K(\widehat{\sigma}_{\varepsilon}, \widehat{\nu})$, conditional on $x$, when $\sigma_{\varepsilon}$ and $\nu$ are unknown.
Thus we can define the brute force simulation estimator of \textsc{CP} as
\begin{align*}
\widehat{\textsc{CP}} = \frac{1}{M} \sum_{k=1}^M \bigg( &{\cal I}\left( |\widehat{g}_{I, k}| \leq z_{1-\alpha/2}, \, |\widehat{h}_k| \leq z_{1-\alpha_H/2} \right) + {\cal I}\left(|\widehat{g}_{J, k}| \leq z_{1-\alpha/2}, \, |\widehat{h}_k| > z_{1-\alpha_H/2} \right) \bigg).
\end{align*}
We can also find an estimator of \textsc{CP} that makes use of a control variate for variance reduction.  This is done as follows.  Let $\textsc{CPK} = P(b \in K(\sigma_{\varepsilon}, \nu) | x)$, the coverage probability of $K(\sigma_{\varepsilon}, \nu)$, conditional on $x$, when $\sigma_{\varepsilon}$ and $\nu$ are known.  Let $g_{I, k}$, $g_{J, k}$ and $h_k$ be the values of $g_I$, $g_J$ and $h$, respectively, on the $k$'th simulation run.  An unbiased estimator of $\textsc{CPK}$ is
\begin{align*}
\widehat{\textsc{CPK}} = \frac{1}{M} \sum_{k=1}^M \bigg( &{\cal I}\left( |g_{I, k}| \leq z_{1-\alpha/2}, \, |h_k| \leq z_{1-\alpha_H/2} \right) + {\cal I}\left(|g_{J, k}| \leq z_{1-\alpha/2}, \, |h_k| > z_{1-\alpha_H/2} \right) \bigg).
\end{align*}
Since we use the same $\varepsilon_{it}^{\dag}$'s and $\eta_i^{\dag}$'s on each simulation run to compute $\widehat{g}_{I, k}$, $\widehat{g}_{J, k}$ and $\widehat{h}_k$, as we do to compute $g_{I, k}$, $g_{J, k}$ and $h_k$, we expect that the correlation between $\widehat{\textsc{CP}}$ and $\widehat{\textsc{CPK}}$ will be close to 1.
We can find $\textsc{CPK}$ exactly using \eqref{cov_prob_known} and Theorem \ref{thm: dist_gih_gjh}.  Therefore an estimator of $\textsc{CP}$ which makes use of a control variate for variance reduction is
\begin{equation*}
\widetilde{\textsc{CP}} = \widehat{\textsc{CP}} - (\widehat{\textsc{CPK}} - \textsc{CPK}),
\end{equation*}
where $(\widehat{\textsc{CPK}} - \textsc{CPK})$ is the control variate which has expected value zero.
We estimate the variance of this estimator by noting that it is an average of iid random variables.


An efficiency analysis performed using the airfare data described in the introduction reveals that, if we find the relative efficiency of the control variate estimator to the brute force estimator for, for example, $\nu \in \{11.3976, 14.3829\}$ (the endpoints of the 98\% confidence interval for $\nu$) and for every $\gamma$ in a grid of values, the minimum gain in efficiency is approximately 10.51 and the maximum gain in efficiency is approximately 18.38.  In other words, we gain efficiency by using the control variate estimator instead of the brute force estimator.

\subsection{Estimation of $\boldsymbol{SEL(\gamma, \nu)}$ for given $\boldsymbol{\gamma}$ and $\boldsymbol{\nu}$ by simulation}
\label{sec_est_sel}

Consider a grid of $\gamma$ values.
Let $\textsc{NUM} =  E \left( (\widehat{\sigma}_{\varepsilon} / \sigma_{\varepsilon}) \left( \widehat{w}^{1/2} {\cal I}({\cal H}) + {\cal I}({\cal H}^c) \right) \, | \, x \right)$ and $\textsc{DENOM} = E(\widehat{\sigma}_{\varepsilon} / \sigma_{\varepsilon})$.
Then the conditional scaled expected length \eqref{scaled_expected_length} is equal to
\begin{equation*}
\frac{ z_{1-\alpha/2}}{\Phi^{-1}\left( (c^* + 1) / 2 \right)} \, \frac{\textsc{NUM}}{\textsc{DENOM}}.
\end{equation*}
By the proof of Theorem \ref{thm_sel_depends}, $\textsc{NUM}$ and $\textsc{DENOM}$ can be expressed in terms of the $\varepsilon_{it}^{\dag}$'s and the $\eta_i^{\dag}$'s, where the $\varepsilon_{it}^{\dag}$'s and $\eta_i^{\dag}$'s are iid $N(0, 1)$.
The simulation method consists of $M$ independent simulation runs.
Define $\textsc{NUM}_k$ to be the value of $(\widehat{\sigma}_{\varepsilon} / \sigma_{\varepsilon}) \left( \widehat{w}^{1/2} \, {\cal I}({\cal H}) + {\cal I}({\cal H}^c) \right)$ and $\textsc{DENOM}_k$ to be the value of $ \widehat{\sigma}_{\varepsilon} / \sigma_{\varepsilon}$ on the $k$'th simulation run  ($k=1, \dots, M$).
For one set of $M$ independent simulation runs we do the following.
On the $k$'th simulation run we generate observations of the $\varepsilon_{it}^{\dag}$'s and the $\eta_i^{\dag}$'s and compute $\textsc{NUM}_k$.
For a second set of $M$ independent simulation runs we do the following.
On the $k$'th simulation run we generate observations of the $\varepsilon_{it}^{\dag}$'s and compute $\textsc{DENOM}_k$.
Now we can define the brute force simulation estimators
\begin{equation*}
\widehat{\textsc{NUM}} = \frac{1}{M} \sum_{k=1}^M \textsc{NUM}_k  \, \, \text{ and } \, \, \widehat{\textsc{DENOM}} = \frac{1}{M} \sum_{k=1}^M \textsc{DENOM}_k.
\end{equation*}

We can also find a control variate estimator of $\textsc{NUM}$.  Define
${\cal B} = \big \{-z_{1-\alpha_H/2} \leq h \leq z_{1-\alpha_H/2} \big \}$ and $\textsc{NUMK} = E\left( w^{1/2} {\cal I}({\cal B}) + {\cal I}({\cal B}^c) \, | \, x\right)$.  Let $\textsc{NUMK}_k$ be the value of $w^{1/2} \, {\cal I}({\cal B}) + {\cal I}({\cal B}^c)$ for the $k$'th simulation run.  Define
\begin{equation*}
\widehat{\textsc{NUMK}} = \frac{1}{M} \sum_{k=1}^M \textsc{NUMK}_k.
\end{equation*}
Now $\textsc{NUMK} = w^{1/2} P({\cal B} \, | \, x) + 1 - P({\cal B} \, | \, x) = 1 + (w^{1/2} - 1) P({\cal B} \, | \, x)$, which can be found exactly using Theorem \ref{thm: dist_gih_gjh}.  Thus an estimator of $\textsc{NUM}$ which makes use of a control variate for variance reduction is
\begin{align*}
\widetilde{\textsc{NUM}} &= \widehat{\textsc{NUM}} - \left( \widehat{\textsc{NUMK}}  - \textsc{NUMK} \right),
\end{align*}
where $\left( \widehat{\textsc{NUMK}}  - \textsc{NUMK} \right)$ is the control variate which has expected value zero.
Therefore an estimator of the scaled expected length that uses a control variate for variance reduction is
\begin{equation*}
\frac{z_{1-\alpha/2}}{\Phi^{-1}((c^*+1)/2)} \frac{\widetilde{\textsc{NUM}}}{\widehat{\textsc{DENOM}}}.
\end{equation*}
%
We estimate the variance of $\widetilde{\textsc{NUM}}$ by noting that it is an average of iid random variables.

Similarly to Section \ref{sec_est_cp}, an efficiency analysis was performed using the airfare data described in the Introduction.  This analysis reveals that, if we find the relative efficiency of the control variate estimator to the brute force estimator
of $\textsc{NUM}$
 for, for example, $\nu \in \{11.3976, 14.3829\}$ (the endpoints of the 98\% confidence interval for $\nu$) and for every $\gamma$ in a grid of values, the minimum gain in efficiency is approximately 0.86 and the maximum gain in efficiency is approximately 10.48.  In other words, we gain efficiency by using the control variate estimator instead of the brute force estimator.


\subsection{Estimation of the density functions of $\boldsymbol{CP(\gamma, \widehat{\nu})}$ and \newline $\boldsymbol{CP(\widehat{\gamma}, \nu)}$ for given $\boldsymbol{\gamma}$ and $\boldsymbol{\nu}$ by simulation}
\label{sec_est_pdfs_CP}

Suppose that $\gamma$ and $\nu$ are specified.
We begin by describing how to estimate the density function of $CP(\gamma, \widehat{\nu})$ by simulation.
Note that
\begin{align}
\label{eqn_pivot}
\frac{\widehat{\nu} + T^{-1}}{\nu + T^{-1}} = \left( N^{-1} \sum_{i=1}^N \left( \frac{ \widetilde{r}_i}{\sigma_{\varepsilon}} \right)^2  \right) \left( \left(N(T-1)\right)^{-1} \sum_{i=1}^N \sum_{t=1}^T \left( \frac{r_{it}}{\sigma_{\varepsilon}} \right)^2 \right) \left( \frac{1}{\nu + T^{-1}} \right).
\end{align}
In the proof of Theorem \ref{thm_cov_depends} we show that $(\widetilde{r}_i / \sigma_{\varepsilon})$ and $(r_{it} / \sigma_{\varepsilon})$ are functions of the $\varepsilon_{it}^{\dag}$'s, $\eta_i^{\dag}$'s, $\nu$ and $x$, where the $\eta_i^{\dag}$'s and $\varepsilon_{it}^{\dag}$'s
are iid $N(0,1)$.
Our simulation consist of $M$ independent simulation runs.
On the $k$'th simulation run ($k=1, \dots, M$) we generate observations of the $\eta_i^{\dag}$'s and the $\varepsilon_{it}^{\dag}$'s and use these to compute an observation of $\left( \widehat{\nu} + T^{-1} \right) / \left( \nu + T^{-1} \right)$
via \eqref{eqn_pivot}.
Thus the $M$ independent simulation runs result in $M$ independent observations of $\left( \widehat{\nu} + T^{-1} \right) / \left( \nu + T^{-1} \right)$.
We transform these observations using the specified value of $\nu$ to obtain $M$ observations of $\widehat{\nu}$.
Now estimate $CP(\gamma, \nu)$, using the simulation method described in Section \ref{sec_est_cp}, for an appropriate grid of equally-spaced values of $\nu$.  Estimates of $CP(\gamma, \nu)$ for values of $\nu$ not on this grid are obtained by linear interpolation.
Use this estimate of $CP(\gamma, \nu)$, considered as a function of $\nu$,
to transform the $M$ observations of $\widehat{\nu}$ into $M$ observations of $CP(\gamma, \widehat{\nu})$. We use the \texttt{density} function in \texttt{R} to approximate the density function of $CP(\gamma, \widehat{\nu})$.

Now we describe how to estimate the density function of $CP(\widehat{\gamma}, \nu)$ by simulation.  It can be shown that $\widehat{\gamma}$, given by \eqref{eqn_gam_hat},
 is a function of the $\varepsilon_{it}^{\dag}$'s, $\eta_i^{\dag}$'s, $\gamma$, $\nu$ and $x$, where the $\eta_i^{\dag}$'s and $\varepsilon_{it}^{\dag}$'s
are iid $N(0,1)$.
Our simulation consist of $M$ independent simulation runs.
On the $k$'th simulation run ($k=1, \dots, M$) we generate observations of the $\eta_i^{\dag}$'s and the $\varepsilon_{it}^{\dag}$'s and use these to compute an observation of $\widehat{\gamma}$.
Now estimate $CP(\gamma, \nu)$, using the simulation method described in Section \ref{sec_est_cp}, for an appropriate grid of equally-spaced values of $\gamma$.  Estimates of $CP(\gamma, \nu)$ for values of $\gamma$ not on this grid are obtained by linear interpolation.
Use this estimate of $CP(\gamma, \nu)$, considered as a function of $\gamma$,
to transform the $M$ observations of $\widehat{\gamma}$ into $M$ observations of $CP(\widehat{\gamma}, \nu)$. We use the \texttt{density} function in R to approximate the density function of $CP(\widehat{\gamma}, \nu)$.

\subsection{Estimation of the quantiles $\boldsymbol{F_{\overline{\alpha}/2}}$ and $\boldsymbol{F_{1 - \overline{\alpha}/2}}$ by simulation}
\label{sec_est_ci_nu}

To construct the equi-tailed $1 - \overline{\alpha}$ confidence interval \eqref{conf_int_nu}
for $\nu$, we need to compute $\widehat{\nu}$ for our data and find
$F_{\overline{\alpha}/2}$ and $F_{1 - \overline{\alpha}/2}$, the $\overline{\alpha}/2$ and $1 - \overline{\alpha}/2$ quantiles, respectively, of the distribution of the pivotal quantity $(\widehat{\nu} + T^{-1}) / (\nu + T^{-1})$.
We estimate these quantiles by simulation as follows.
Our simulation consist of $M$ independent simulation runs.
Since $(\widehat{\nu} + T^{-1}) / (\nu + T^{-1})$ is a pivotal quantity we specify an arbitrary value of $\nu$, say $\nu = 1$.
Similarly to the previous section, on the $k$'th simulation run (for $k=1, \dots, M$) we generate observations of the $\eta_i^{\dag}$'s and the $\varepsilon_{it}^{\dag}$'s and use these to compute an observation of $\left( \widehat{\nu} + T^{-1} \right) / \left( \nu + T^{-1} \right)$.
Thus $M$ independent simulation runs result in $M$ observations of $\left( \widehat{\nu} + T^{-1} \right) / \left( \nu + T^{-1} \right)$.
Arrange these observations in increasing order.
Estimate the $p$'th quantile of the distribution of $\left( \widehat{\nu} + T^{-1} \right)/\left( \nu + T^{-1} \right)$ by the $r$'th ordered estimate, where $r$ and $M$ are chosen such that $p = r / (M + 1)$.
For example, to estimate the 0.01'th quantile of the distribution of $\left(\widehat{\nu} + T^{-1} \right)/\left(\nu + T^{-1} \right)$  we may choose $r = 100$ and $M = 9999$.
A confidence interval for $\nu$ is now found using \eqref{conf_int_nu}, where $F_{1-\overline{\alpha}/2}$ and $F_{\overline{\alpha}/2}$ are replaced by the relevant quantile estimates obtained by simulation.

\begin{appendix}

\section*{Appendix}

\subsection*{The estimators of $\boldsymbol{b_B}$ and $\boldsymbol{b_W}$}

We can write model \eqref{CRE_model} in matrix form as $Y = X \beta + u$, where $Y$, $\beta$ and $u$ are the column vectors $Y = (y_{11}, \dots, y_{1T}, \dots, y_{N1}, \dots, y_{NT})$, $\beta = (a, b_W, b_B)$ and $u = (\eta_1 + \varepsilon_{11}, \dots, \eta_1 + \varepsilon_{1T}, \dots, \eta_N + \varepsilon_{N1}, \dots, \eta_N + \varepsilon_{NT})$,
and $X$ is the matrix
%
\begin{equation*}
X = \begin{bmatrix} 1 & \dots & 1 & \dots & 1 & \dots & 1 \\ x_{11} - \overline{x}_1 & \dots & x_{1T} - \overline{x}_1 & \dots & x_{N1} - \overline{x}_N & \dots & x_{NT} - \overline{x}_N \\ \overline{x}_1 & \dots & \overline{x}_1 & \dots & \overline{x}_N & \dots & \overline{x}_N \end{bmatrix}'.
\end{equation*}
Let $I_T$ denote the $T \times T$ identity matrix and $e_T$ be a $T$ column vector of 1's.  The GLS estimator of $\beta$ is given by
$\big(X' L(\nu)^{-1} X \big)^{-1} X' L(\nu)^{-1} Y$, where $L(\nu)^{-1}$ is an $NT \times NT$ block diagonal matrix with
$N$ identical block diagonal elements each of which is the $T \times T$ matrix $I_T - (\nu / (1 + \nu T)) \, e_T \, e_T'$.  Using this, it can be shown that the GLS estimators of $b_W$ and $b_B$ are
\begin{align}
\widetilde{b}_W &= \frac{ \sum_{i=1}^N \sum_{t=1}^T (x_{it} - \overline{x}_i)(y_{it} - \overline{y}_i)}{\SSW}  = b_W + \frac{ \sum_{i=1}^N \sum_{t=1}^T (x_{it} - \overline{x}_i)(\varepsilon_{it} - \overline{\varepsilon}_i)}{\SSW}  \label{est_b_W}  \\[1.2ex]
\widetilde{b}_B &= \frac{\sum_{i=1}^N (\overline{x}_i - \overline{x})(\overline{y}_i - \overline{y})}{\SSB} = b_B + \frac{ \sum_{i=1}^N (\overline{x}_i - \overline{x})((\eta_i - \overline{\eta}) + (\overline{\varepsilon}_i - \overline{\varepsilon}))}{\SSB}. \label{est_b_B}
\end{align}
Remember, as defined in Section 4, $\SSB = \sum_{i=1}^N (\overline{x}_i - \overline{x})^2$ and
$\SSW = \sum_{i=1}^N \sum_{t=1}^T (x_{it} - \overline{x}_i)^2$.
The OLS estimator of $b_W$ based on model \eqref{FE_model} is \eqref{est_b_W} and the OLS estimator of $b_B$ based on model \eqref{BE_model} is \eqref{est_b_B}.
Note that the covariance matrix of the GLS estimator of $\beta$ is given by $\sigma_{\varepsilon}^2 \big( X' L(\nu)^{-1} X \big)^{-1}$.

\subsection*{The random variables that determine the conditional \newline coverage probability}

Let $\widehat{g}_I$, $\widehat{g}_J$ and $\widehat{h}$ denote the random variables $g_I$, $g_J$ and $h$ introduced in Section \ref{sec_cp_est's_known}, when $\sigma_{\varepsilon}$ and $\nu$ have been replaced by their estimators.  By Section \ref{cp_pretest_interval}, $P\left(b \in K(\widehat{\sigma}_{\varepsilon}, \widehat{\nu}) \,  | \, x \right)$ is equal to
\begin{align*}
& P\left( b \in I(\widehat{\sigma}_{\varepsilon}, \widehat{\nu}), \, H(\widehat{\sigma}_{\varepsilon}, \widehat{\nu}) \leq z_{1-\alpha_H/2}^2 \, | \, x \right) + P\left( b \in J(\widehat{\sigma}_{\varepsilon}), \,  H(\widehat{\sigma}_{\varepsilon}, \widehat{\nu}) > z_{1-\alpha_H/2}^2 \, | \, x \right) \\
& = P\left( |\widehat{g}_I| \leq z_{1-\alpha/2}, \, |\widehat{h}| \leq z_{1-\alpha_H/2}^2 \, | \, x \right) + P\left( |\widehat{g}_J| \leq z_{1-\alpha/2}, \, |\widehat{h}| > z_{1-\alpha_H/2} \, | \, x \right).
\end{align*}

The following lemma gives expressions for $\widehat{g}_I$, $\widehat{g}_J$ and $\widehat{h}$ that are used
in the proofs of the theorems.

\begin{lemma}
\label{lem_long_gI_gJ_h}
\begin{align*}
\widehat{g_I} &= \frac{ q(\widehat{\nu}, T) \sum_{i=1}^N \sum_{t=1}^T (x_{it} - \overline{x}_i)(\varepsilon_{it} - \overline{\varepsilon}_i) + \sum_{i=1}^N (\overline{x}_i - \overline{x})\left((\eta_i - \overline{\eta}) + (\overline{\varepsilon}_i - \overline{\varepsilon})\right)}{ \widehat{\sigma}_{\varepsilon} \left( q(\widehat{\nu}, T) \, \SSW \,  (q(\widehat{\nu}, T) + r(x)) \right)^{1/2}}      \\[1.2ex]
& \indent + \gamma  \left( \frac{r(x)}{ q(\widehat{\nu}, T) \, (q(\widehat{\nu}, T) + r(x)) } \right)^{1/2} \left( \frac{\SSB}{N} \right)^{1/2} \left( \frac{\sigma_{\varepsilon}}{\widehat{\sigma}_{\varepsilon}} \right)\\
\\
\widehat{g}_J &= \frac{ \sum_{i=1}^N \sum_{t=1}^T (x_{it} - \overline{x}_i)(\varepsilon_{it} - \overline{\varepsilon}_i)}{\widehat{\sigma}_{\varepsilon} \left( \SSW \right)^{1/2} }\\
\\
\widehat{h} &= \frac{r(x) \sum_{i=1}^N \sum_{t=1}^T (x_{it} - \overline{x}_i)(\varepsilon_{it} - \overline{\varepsilon}_i) - \sum_{i=1}^N (\overline{x}_i - \overline{x}) \left( (\eta_i - \overline{\eta}) + (\overline{\varepsilon}_i - \overline{\varepsilon}) \right)}{\widehat{\sigma}_{\varepsilon} \left(\SSB \,  (q(\widehat{\nu}, T) + r(x)) \right)^{1/2}} \\[1.2ex]
&\indent - \gamma \left( \frac{1}{r(x) + q(\widehat{\nu}, T)} \right)^{1/2} \left( \frac{\SSB}{N} \right)^{1/2} \left( \frac{\sigma_{\varepsilon}}{\widehat{\sigma}_{\varepsilon}} \right).
\end{align*}
\end{lemma}

\begin{proof}
It can be shown that
\begin{align}
\label{Dist_BtildeW_BtildeB}
\widetilde{b}_W \sim N\left( b_W, \, \frac{\sigma_{\varepsilon}^2}{\SSW} \right) \, \,  \text{ and } \, \,
\widetilde{b}_B \sim N\left( b_B, \, \frac{ \sigma^2_{\varepsilon} (\nu + T^{-1})}{\SSB} \right).
\end{align}
Thus $\text{Var}(\widetilde{b}_W \, | \, x) = \sigma_{\varepsilon}^2 / \SSW$ and $\text{Var}(\widetilde{b}_B \, | \, x ) = \left( \sigma^2_{\varepsilon}(\nu + T^{-1}) \right) / \SSB$.
Lemma \ref{lem_long_gI_gJ_h} follows from this, equations \eqref{est_b_W} and \eqref{est_b_B}, and Maddala's (1971) equality $\widehat{b} = \widehat{w} \, \widetilde{b}_W + (1 - \widehat{w}) \, \widetilde{b}_B$.
\end{proof}

\subsection*{Proof of Theorem \ref{thm_cov_depends}}

Let $\varepsilon_{it}^{\dag} = \varepsilon_{it}/\sigma_{\varepsilon}$, $\overline{\varepsilon}_i^{\dag} = \overline{\varepsilon}_i/\sigma_{\varepsilon}$, $\overline{\varepsilon}^{\dag} = \overline{\varepsilon}/\sigma_{\epsilon}$, $\eta_i^{\dag} = \eta_i/\sigma_{\eta}$ and $\overline{\eta}^{\dag} = \overline{\eta}/\sigma_{\eta}$.
The $\eta_i^{\dag}$'s and
$\varepsilon_{it}^{\dag}$'s are iid $N(0, 1)$.
It can be shown that
\begin{align*}
\frac{r_{it}}{\sigma_{\varepsilon}} &= - \frac{(x_{it} - \overline{x}_i) \sum_{i=1}^N \sum_{t=1}^T (x_{it} - \overline{x}_i)(\varepsilon_{it}^{\dag} - \overline{\varepsilon}_i^{\dag})}{\SSW} + (\varepsilon^{\dag}_{it} - \overline{\varepsilon}^{\dag}_i) \\[1.2ex]
\frac{\widetilde{r}_i}{\sigma_{\varepsilon}} &= - \left( \frac{ (\nu^{1/2} \, \overline{\eta}^{\dag} + \overline{\varepsilon}^{\dag}) \sum_{i=1}^N \overline{x}_i^2 - \overline{x} \sum_{i=1}^N \overline{x}_i (\nu^{1/2} \,\eta_i^{\dag} + \overline{\varepsilon}_i^{\dag})}{\SSB} \right) \\
& \indent - \left( \frac{ \sum_{i=1}^N (\overline{x}_i - \overline{x})\left( \nu^{1/2} \, (\eta_i^{\dag} - \overline{\eta}^{\dag}) + (\overline{\varepsilon}_i^{\dag} - \overline{\varepsilon}^{\dag}) \right)}{\SSB} \right) \overline{x}_i + \nu^{1/2} \, \eta_i^{\dag} + \overline{\varepsilon}_i^{\dag}.
\end{align*}
Thus $\widehat{\sigma}_{\varepsilon} / \sigma_{\varepsilon}$ and $\widehat{\sigma}_{\eta} / \sigma_{\varepsilon}$ can be expressed in terms of the $\varepsilon_{it}^{\dag}$'s, $\eta_i^{\dag}$'s, $\nu$ and $x$.
Hence $\widehat{\nu} = \widehat{\sigma}_{\eta} / \widehat{\sigma}_{\varepsilon}$ can be expressed in terms of the $\varepsilon_{it}^{\dag}$'s, $\eta_i^{\dag}$'s, $\nu$ and $x$.
Now, using Lemma \ref{lem_long_gI_gJ_h}, divide the numerator and denominator of $\widehat{g}_I$, $\widehat{g}_J$ and $\widehat{h}$ by $\sigma_{\varepsilon}$.  It follows that $\widehat{g}_I$, $\widehat{g}_J$ and $\widehat{h}$ can be expressed in terms of $\varepsilon_{it}^{\dag}$'s, $\eta_i^{\dag}$'s, $\nu$, $\gamma$ and $x$.

\subsection*{Proof of Theorem \ref{thm_cov_even}}

Define $\varepsilon = (\varepsilon_{11}, \dots, \varepsilon_{1J}, \dots, \varepsilon_{N1}, \dots, \varepsilon_{NJ})$ and
$\eta= (\eta_1, \dots, \eta_N)$.
Introduce the notation $\widehat{g}_I = \widehat{g}_I(x, \varepsilon, \eta, \gamma)$, $\widehat{g}_J = \widehat{g}_J(x, \varepsilon)$ and $\widehat{h} = \widehat{h}(x, \varepsilon, \eta, \gamma)$ to show the dependence of $\widehat{g}_I$, $\widehat{g}_J$ and $\widehat{h}$ on $x, \varepsilon, \eta$ and $\gamma$.
Note that $\widehat{g}_J(x, \varepsilon)$ is not a function of $\gamma$.
We have $\{ |\widehat{g}_I| \leq z_{1-\alpha/2} \} = \{ - z_{1-\alpha/2} \leq - \widehat{g}_I \leq z_{1-\alpha/2} \}$, $ \{ |\widehat{h}| \leq z_{1-\alpha_H/2} \} = \{ -z_{1-\alpha_H/2} \leq - \widehat{h} \leq z_{1-\alpha_H/2} \}$ and $\{ |\widehat{h}| > z_{1-\alpha_H/2} \} = \{ -\widehat{h} < z_{1-\alpha_H/2} \} \cup \{ -\widehat{h} > z_{1-\alpha_H/2} \}$.
Thus for the coverage probability to be an even function of $\gamma$, it is sufficient to prove that
(a) the distribution of $\big( \widehat{g}_I(x, \varepsilon, \eta, d), \widehat{h}(x, \varepsilon, \eta, d) \big)$ is the same as the distribution of $\big(-\widehat{g}_I(x, \varepsilon, \eta, -d), -\widehat{h}(x, \varepsilon, \eta, -d) \big)$ and (b)
the distribution of $\big( \widehat{g}_J(x, \varepsilon), \widehat{h}(x, \varepsilon, \eta, d) \big)$ is the same as the distribution of $\big(-\widehat{g}_J(x, \varepsilon), -\widehat{h}(x, \varepsilon, \eta, -d) \big)$.
For the sake of brevity, we
only give the proof of (a). The proof of (b) is similar.

Let $\varepsilon_{it}^* = -\varepsilon_{it}$ for $i=1, \dots, N$ and $t=1, \dots, T$, and \newline
$\varepsilon^* = (-\varepsilon_{11}, \dots, -\varepsilon_{1J}, \dots, -\varepsilon_{N1}, \dots, -\varepsilon_{NJ})$.
Also let $\eta_i^* = - \eta_i$ for $i=1, \dots, N$ and $\eta^* =(-\eta_1, \dots, -\eta_N)$.
Since the $\varepsilon_{it}$'s and $\eta_i$'s are independent, the
$\varepsilon_{it}$'s are iid $N(0, \sigma^2_{\varepsilon})$ and the $\eta_i$'s are iid $N(0, \sigma^2_{\eta})$, it follows that the
$\varepsilon_{it}^*$'s and $\eta_i^*$'s are independent, the
$\varepsilon_{it}^*$'s are iid $N(0, \sigma^2_{\varepsilon})$ and the $\eta_i^*$'s are iid $N(0, \sigma^2_{\eta})$.
Therefore $\big( \widehat{g}_I(x, \varepsilon^*, \eta^*, d), \widehat{h}(x, \varepsilon^*, \eta^*, d) \big)$
has the same distribution as $\big( \widehat{g}_I(x, \varepsilon, \eta, d), \widehat{h}(x, \varepsilon, \eta, d) \big)$.
Using Lemma \ref{lem_long_gI_gJ_h}, it can be shown that
$\big(-\widehat{g}_I(x, \varepsilon, \eta, d), -\widehat{h}(x, \varepsilon, \eta, d) \big)$ has the same distribution as
$\big(\widehat{g}_I(x, \varepsilon^*, \eta^*, -d), \widehat{h}(x, \varepsilon^*, \eta^*, -d) \big)$.
It follows that the distribution of
$\big(\widehat{g}_I(x, \varepsilon, \eta, d), \widehat{h}(x, \varepsilon, \eta, d) \big)$ is the same as the distribution of \newline
$\big(-\widehat{g}_I(x, \varepsilon, \eta, -d),-\widehat{h}(x, \varepsilon, \eta, -d) \big)$.

\subsection*{Proof of Theorem \ref{thm_pivot}}

The following lemma defines new quantities whose joint distribution does not depend on any unknown parameters.

\begin{lemma}
\label{dist_tau_phi}
For $i = 1, \dots, N$ and $t = 1, \dots, T$, define
\begin{equation*}
\vartheta_{it}^{\dag} =  \frac{ \varepsilon_{it} - \overline{\varepsilon}_i}{\sigma_{\varepsilon}} \indent \text{ and } \indent  \varphi_i^{\dag} = \frac{ \eta_i + \overline{\varepsilon}_i}{\left( \sigma^2_{\varepsilon}\left( \nu + J^{-1} \right) \right)^{1/2} }.
\end{equation*}
Then the random vector $(\varphi_1^{\dag}, \dots, \varphi_N^{\dag}, \vartheta_{11}^{\dag}, \dots, \vartheta_{1T}^{\dag}, \dots, \vartheta_{N1}^{\dag}, \dots, \vartheta_{NT}^{\dag})$ has a multivariate normal distribution with mean vector 0 and covariance matrix which does not depend on any unknown parameters.
\end{lemma}

\begin{proof}
We assume that the $\varepsilon_{it}$'s and $\eta_i$'s are independent, that the
$\varepsilon_{it}$'s are iid $N(0, \sigma^2_{\varepsilon})$ and that the $\eta_i$'s are iid $N(0, \sigma^2_{\eta})$.
Thus $E(\varepsilon_{it} - \overline{\varepsilon}_i) = 0$, $\text{Cov}(\varepsilon_{it}, \overline{\varepsilon}_i) = T^{-1} \, \sigma^2_{\varepsilon}$, $\text{Var}(\varepsilon_{it} - \overline{\varepsilon}_i) = \sigma^2_{\varepsilon}(1 - T^{-1})$, $E(\eta_i + \overline{\varepsilon}_i) = 0$, $\text{Cov}(\eta_i, \, \overline{\varepsilon}_i) = 0$ and $\text{Var}(\eta_i + \overline{\varepsilon}_i) = \sigma^2_{\varepsilon} \, (\nu + T^{-1})$.
It follow that $E(\vartheta_{it}^{\dag}) = 0$, $E(\varphi_i^{\dag}) = 0$, $\text{Var}(\vartheta_{it}^{\dag}) = 1 - T^{-1}$ and $\text{Var}(\varphi_i^{\dag}) = 1$.
Also, for $i \not= k$ and $t \not= s$, we have $\text{Cov}(\vartheta_{it}^{\dag}, \vartheta_{is}^{\dag}) = T^{-1}$, $\text{Cov}(\vartheta_{it}^{\dag}, \, \vartheta_{kt}^{\dag}) = 0$, $\text{Cov}(\varphi_i^{\dag}, \varphi_k^{\dag}) = 0$, $\text{Cov}(\varphi_i^{\dag}, \vartheta_{it}^{\dag}) = 0$ and $\text{Cov}(\varphi_k^{\dag}, \vartheta_{it}^{\dag}) = 0$.
\end{proof}

\noindent The denominator of \eqref{nu_hat_over_nu} can be written as
\begin{equation*}
\frac{\sum_{i=1}^N \sum_{t=1}^T (\vartheta_{it}^{\dag})^2}{N(T-1)} - \frac{\left( \sum_{i=1}^N \sum_{t=1}^T (x_{it} - \overline{x}_i) \, \vartheta_{it}^{\dag} \right)^2}{N(T-1) \, \SSW},
\end{equation*}
which is in terms of the $\vartheta_{it}^{\dag}$'s and $x$.  In a similar manner, the numerator of \eqref{nu_hat_over_nu} can be written in terms of the $\vartheta_{it}^{\dag}$'s, $\varphi_i^{\dag}$'s and $x$.  It follows from Lemma \ref{dist_tau_phi}
that the distribution of \eqref{nu_hat_over_nu} does not depend on any unknown parameters.

\subsection*{Proof of Theorem \ref{thm_sel}}

It can be shown that the length of $I(\widehat{\sigma}_{\varepsilon}, \widehat{\nu})$ is $2 \, z_{1-\alpha/2} \, \widehat{\sigma}_{\varepsilon} \, \SSW^{-1/2} \, \widehat{w}^{1/2}$ and that the length of $J(\widehat{\sigma}_{\varepsilon})$ is $2 \, z_{1-\alpha/2} \, \widehat{\sigma}_{\varepsilon} \, \SSW^{-1/2}$.  Hence the length of $K(\widehat{\sigma}_{\varepsilon}, \widehat{\nu})$ is equal to $2 \, z_{1-\alpha/2} \, \widehat{\sigma}_{\varepsilon} \, \SSW^{-1/2} \left( \widehat{w}^{1/2} \, {\cal I}({\cal H}) + {\cal I}({\cal H}^c) \right)$.  Also, the length of $J_{c^*}(\widehat{\sigma}_{\varepsilon})$ is equal to $2 \, \Phi^{-1}\left((c^* + 1)/2 \right) \widehat{\sigma}_{\varepsilon} \SSW^{-1/2}$.  Thus the scaled expected length is equal to
\begin{equation*}
\frac{E \left( 2 \, z_{1-\alpha/2} \, \widehat{\sigma}_{\varepsilon} \, \SSW^{-1/2} \left( \widehat{w}^{1/2} {\cal I}({\cal H}) + {\cal I}({\cal H}^c) \right) \, | \, x \right)}{E \left(2 \, \Phi^{-1}\left((c^* + 1)/2 \right) \widehat{\sigma}_{\varepsilon} \, \SSW^{-1/2} \right) },
\end{equation*}
which can be simplified to obtain \eqref{scaled_expected_length}.

\subsection*{Proof of Theorem \ref{thm_sel_depends}}

Consider the conditional scaled expected length expression given by \eqref{scaled_expected_length}.
In the proof of Theorem 1 we have shown that $(\widehat{\sigma}_{\varepsilon} / \sigma_{\varepsilon})$ is a function of the $\varepsilon_{it}^{\dag}$'s and $x$.
Next, $c^*$ depends on $1 - \overline{\alpha}$, and ${\cal I}({\cal H})$ and ${\cal I}({\cal H}^c)$ depend on $\alpha_H$ and $\widehat{h}$.
It follows from the proof of Theorem 1 that $\widehat{h}$ is a function of the $\varepsilon_{it}^{\dag}$'s, $\eta_i^{\dag}$'s, $\nu$, $\gamma$ and $x$.
We have also shown in the proof of Theorem 1 that $\widehat{\nu}$ is a function of the $\varepsilon_{it}^{\dag}$'s, $\eta_i^{\dag}$'s, $\nu$ and $x$, which implies that $\widehat{w} = q(\widehat{\nu}, T) / (q(\widehat{\nu}, T) + r(x))$  is a function of the $\varepsilon_{it}^{\dag}$'s, $\eta_i^{\dag}$'s, $\nu$ and $x$.

\subsection*{Proof of Theorem \ref{thm_sel_even}}

The conditional scaled expected length \eqref{scaled_expected_length} depends on $\gamma$ only through ${\cal I}({\cal H})$ and ${\cal I}({\cal H}^c)$.  Similarly to the proof of Theorem 2, we write $\widehat{h} = \widehat{h}(x, \varepsilon, \eta, \gamma)$ to emphasize the dependence of $\widehat{h}$ on $x$, $\varepsilon$, $\eta$ and $\gamma$.  In the proof of Theorem 2 we have shown that the distribution of $\widehat{h}(x, \varepsilon, \eta, d)$ is the same as the distribution of $-\widehat{h}(x, \varepsilon, \eta, -d)$.  It follows that both ${\cal I}({\cal H})$ and ${\cal I}({\cal H}^c)$ are even functions of $\gamma$.  Therefore the conditional scaled expected length is an even function of $\gamma$.

\subsection*{Proof of Theorem \ref{thm: dist_gih_gjh}}

It follows from \eqref{Dist_BtildeW_BtildeB}, Maddala's (1971) equality and the fact that $\text{Cov}(\widetilde{b}_W, \widetilde{b}_B \, | \, x) = 0$, that
\begin{align*}
\widehat{b} - b &\sim N \left( \frac{ \xi \, r(x)}{q(\nu, T) + r(x)}, \,  \frac{\sigma^2_{\varepsilon}}{\SSW} \frac{q(\nu, T) }{q(\nu, T)+ r(x) }  \right), \label{bhat_b_dist} \\[1.2ex]
\widetilde{b}_W - b &\sim N\left(0, \, \frac{\sigma^2_{\varepsilon}}{\SSW} \right),  \\[1.2ex]
\widetilde{b}_W - \widetilde{b}_B &\sim N\left( -\xi, \, \frac{\sigma^2_{\varepsilon}}{\SSW} \left( 1 + \dfrac{q(\nu, T)}{r(x)} \right) \right).
\end{align*}
Now use the definitions of $g_I$, $g_J$ and $h$ and find the relevant expected values, variances and covariances.

\end{appendix}

\section*{References}

\rf Aldrich, J., 2005. Fisher and regression. Statistical Science, 20, 401--417.

\smallskip

\rf Cox, D.R., 2006. Principles of Statistical Inference. Cambridge University Press: Cambridge.

\smallskip

\rf Guggenberger, P., 2010. The impact of a Hausman pretest on the size of a hypothesis test: the panel data case. Journal of Econometrics, 156, 337--343.

\smallskip

\rf Hausman, J. A., 1978. Specification test in econometrics. Econometrica, 46, 1251--1271.

\smallskip

\rf Kabaila, P., Mainzer, R., Farchione, D., 2015. The impact of a Hausman pretest, applied to panel data, on the coverage probability of
confidence intervals. Economics Letters, 131, 12--15.

\smallskip

\rf Koopmans, T.C., 1937. Linear Regression Analysis of Economic Time Series. Bohn, Haarlem, Netherlands.

\smallskip

\rf Maddala, G. S., 1971.  The use of variance components models in pooling cross section and time series data.  Econometrica, 39, 341 -- 358.

\smallskip

\rf Wooldridge, J. M., 2013.  Introductory econometrics: a modern approach.  5th edition.   South-Western: Ohio.

\end{document}